\documentclass[11pt]{amsart} 
\usepackage{amssymb}
\usepackage{latexsym}
\usepackage{amsmath}

\begin{document}

\numberwithin{equation}{section}          

\theoremstyle{plain}
\newtheorem{theorem}{Theorem}
\newtheorem{lemma}{Lemma}
\newtheorem{corollary}{Corollary}
\newtheorem{proposition}{Proposition}

\theoremstyle{definition}
\newtheorem{definition}{Definition}
\newtheorem{remark}{Remark}
\newtheorem{example}{Example}

\author[S. Pilipovi\'{c}]{Stevan Pilipovi\'{c}}
\address{Department of Mathematics and Informatics\\ University of Novi Sad\\ Trg Dositeja Obradovi\' ca 4\\ 21000 Novi Sad, Serbia}
\email{pilipovics@yahoo.com}

\author[D. Scarpal\'{e}zos]{Dimitris Scarpal\'{e}zos}

\address{Centre de Math\' ematiques de Jussieu\\
Universit\' e Paris 7 Denis Diderot\\ Case Postale 7012, 2, place
Jussieu\\ F-75251 Paris Cedex 05, France}
\email{scarpa@math.jussieu.fr}

\author[J. Vindas]{Jasson Vindas}

\address{Department of Mathematics\\
Ghent University\\
Krijgslaan 281 Gebouw S22\\
B-9000 Gent, Belgium}
\email{jvindas@cage.Ugent.be}

\title{Classes of Generalized Functions with Finite Type Regularities}


\thanks{The work of S. Pilipovi\'{c} is supported by the Serbian Ministry
of Education and Science, through 
project number 174024. J. Vindas gratefully acknowledges support by a Postdoctoral Fellowship of the Research Foundation--Flanders (FWO, Belgium)}


\begin{abstract}
We introduce and analyze spaces and algebras of generalized
functions which correspond to H\" older, Zygmund, and Sobolev spaces of functions. The main scope of the paper is the  characterization of the regularity
of distributions that are embedded into the corresponding space or algebra of
generalized functions with finite type regularities.\end{abstract}

\maketitle

\section{Introduction}\label{sec0}
In this paper we develop  regularity theory in  generalized
function algebras parallel to the corresponding theory within
distribution spaces. We consider subspaces or subalgebras in
algebras of generalized functions which correspond to the classical Sobolev
spaces $W^{k,p}$, Zygmund spaces $C_*^s$, and H\" older spaces
$\mathcal H^{k,\tau}$. We refer to 
\cite{col1,gkos,ober001} for the theory of
generalized function algebras and their use in the study of
various classes of equations.

It is known that the elements of  algebras of generalized functions
are represented by nets $(f_\varepsilon)_\varepsilon$ of smooth
functions, with appropriate growth as $\varepsilon\rightarrow 0,$ that the spaces of Schwartz distributions are embedded
into the corresponding algebras, and that  the algebra of regular generalized functions
corresponding to the space of smooth
functions 
is $\mathcal G^\infty$ (cf. \cite{ober001,Ver}). Intuitively, these
algebras are obtained through regularization of distributions
(convolving them with delta nets) and 
factorization of an appropriate algebra of moderate nets of smooth functions with respect to an ideal of negligible nets, as Colombeau did
 \cite{col1}  with his algebra $\mathcal G(\mathbb R^d)$ (in this way the name Colombeau algebras has appeared). By construction distributions are included  in the corresponding Colombeau algebras and their natural linear operations are preserved.
  
 The main goal of this paper is to find out natural conditions with respect to the
growth order in $\varepsilon$ which characterize generalized function spaces and algebras with finite type regularities. Actually, our main task is to seek optimal definitions for such generalized function spaces, since we would like to have backward information on the regularity properties of Schwartz distributions that are embedded into the corresponding space of generalized functions. Sobolev and Zygmund type spaces are very suitable for this purpose. In particular, the Zygmund type spaces are useful in this respect, since we can almost literary transfer classical properties of these spaces into their generalized versions.

One can find many articles in the literature where local and
microlocal properties of generalized functions in generalized
function algebras have been considered; besides the quoted monographs, we
refer to the papers 
\cite{gah,herhup,her,kuhe,ko,psv2,Ver}. The motivation of this article came partly from the papers
\cite{herhup,her}, where Zygmund type algebras of
generalized functions were studied and used in the qualitative
analysis of certain hyperbolic problems. We shall define new classes of generalized functions that are also intrinsically connected with the classical Zygmund spaces.

Note that in our paper \cite{piscavi} we have studied regularity properties of distributions $T$ in terms of growth properties of regularizing sequences $T*\delta_n$ with respect to the parameter $n\in\mathbb{N}$ and various seminorms. Some ideas from that paper are implicitly employed in Section \ref{cass} of the present article, where we reinterpret them in the setting of the Colombeau algebra.

The paper is organized as follows. 
Preliminaries are given in Section \ref{gdef}. In Section \ref{classes} we define our new spaces of generalized functions with finite type regularities, the spaces $\mathcal G^{k,-s}$,
which correspond to local versions of the Zygmund  spaces $C^r_*$. They are subspaces of the Colombeau algebra $\mathcal G(\Omega)$. Then, in Section \ref{cass}, we investigate the role of these new classes of generalized functions in the regularity analysis of distributions; we characterize the regularity properties of those distributions that, after embedding or association, belong to one of these classes.  Our main result is Theorem \ref{1mtheorem} of Section \ref{cass}, we show
that the intersection of $\mathcal G^{k,-s}(\Omega)$ with the embedded image of $\mathcal{D}'(\Omega)$ is precisely $\iota(C_{*,loc}^{k-s}(\Omega))$. As a consequence, we obtain a quick proof of the important regularity theorem for $\mathcal G^\infty$ (\cite{ober001}). Theorems \ref{regas} and \ref{mtheorem} deal with the analysis of the regularity of a distribution via strong versions of 
association. Finally, Section \ref{szh} is devoted to the study of some generalized function spaces and algebras that are very helpful in \emph{global} regularity analysis. The global classes that we introduce are capable of recovering the embedded image of the classical Zygmund and H\" older spaces of functions. We also compare our global Zygmund type generalized function spaces with the one proposed in \cite{herhup,her}.

\section{Preliminaries and notation}
\label{gdef}
We denote by $\Omega$ an open subset of $\mathbb{R}^{d}$.
We consider the families of local Sobolev seminorms
$||\rho||_{W^{m,p}(\omega)}= \sup
\{||\rho^{(\alpha)}||_{L^{p}(\omega)};\:|\alpha|\leq m \},$ where $m\in
{\mathbb N}_0$, $p\in[1,\infty]$, and $\omega$ runs over all open subsets of $\Omega$
with compact closure ($\omega\subset\subset\Omega$). The local Sobolev space is then denoted as $W^{m,p}_{loc}(\Omega)$. 
In case $\omega$ is replaced by $\Omega,$ we obtain the family of norms
$||\:\cdot\:||_{W^{m,p}(\Omega)}, m\in\mathbb N_0$.

Let $\mathcal{E}(\Omega)$ be the space of smooth functions in
$\Omega.$ The spaces of moderate nets and negligible nets $\mathcal
E_{L_{loc}^p,M}(\Omega)$ and $\mathcal  N_{L_{loc}^p}(\Omega)$
consist, resp., of nets $(f_{\varepsilon
})_{\varepsilon\in(0,1]}=(f_{\varepsilon
})_{\varepsilon}\in \mathcal{E}(\Omega)^{(0,1]}$ with the
properties
\begin{equation}\label{2drs}
(\forall m\in\mathbb{N}_0)(\forall
\omega\subset\subset\Omega)(\exists a\in\mathbb{R})
(||f_{\varepsilon}||_{W^{m,p}(\omega)}=O(\varepsilon^{a})),
\end{equation}
\[
\mbox{ resp., } \;(\forall m\in\mathbb{N}_0)(\forall
\omega\subset\subset \Omega)(\forall b\in\mathbb{R})
(||f_{\varepsilon}||_{W^{m,p}(\omega)}=O(\varepsilon^{b}))
\]
(big $O$ and small $o$ are the Landau symbols). 
Note that, for $p\in[1,\infty]$,
$$\mathcal
{E}_{M}(\Omega):=\mathcal{E}_{L_{loc}^\infty,M}(\Omega)=\mathcal
E_{L_{loc}^p,M}(\Omega),\;\; \mathcal
{N}(\Omega)=\mathcal
N_{L_{loc}^\infty}(\Omega)=\mathcal N_{L_{loc}^p}(\Omega).$$
We obtain
the Colombeau algebra of generalized functions as a quotient:
$$\mathcal
G(\Omega)=\mathcal G_{L_{loc}^p}(\Omega)=\mathcal{E}_{L_{loc}^p,M}(\Omega)/\mathcal N_{L_{loc}^p}(\Omega), \ p\in[1,\infty].$$

The embedding of the Schwartz distribution space
$\mathcal{E}^{\prime}(\Omega)$ into $\mathcal{G}(\Omega)$ is realized through the sheaf
homomorphism $ \mathcal{E}^{\prime}(\Omega)\ni T\mapsto\iota(T)=
[((T\ast\phi_{\varepsilon})_{|{\Omega}})_\varepsilon]\in\mathcal{G}(\Omega)$,
where the fixed net of mollifiers
$(\phi_{\varepsilon})_{\varepsilon}$ is defined by
$\phi_{\varepsilon}=\varepsilon ^{-d}\phi(\cdot/\varepsilon),$ 
$0<\varepsilon\leq 1,$ and $\phi\in\mathcal{S}(\mathbb{R}^{d})$ satisfies
$$\int_{\mathbb{R}^{d}}
\phi(t)dt=1,\;\int_{\mathbb{R}^{d}}
t^{\alpha}\phi(t)dt=0,\  \ \ |\alpha|>0.$$
This sheaf homomorphism \cite{gkos}, extended over $\mathcal{D}^{\prime}$,
gives the embedding of $\mathcal{D}^{\prime}(\Omega)$ into
$\mathcal{G}(\Omega)$. 
We also use the notation $\iota$ for the mapping from $\mathcal{E}'(\Omega)$
into $ \mathcal{E}_M(\Omega)$, $\iota(T)=
((T\ast\phi_{\varepsilon})_{|{\Omega}})_\varepsilon$. Throughout this article, $\phi$ will \emph{always} be fixed and satisfy the above condition over its moments. 

The generalized algebra of regular generalized functions
$\mathcal{G}^\infty(\Omega)$ is defined in \cite{ober001} as the
quotient of the algebras $\mathcal{E}^\infty_{M}(\Omega)$ and
$\mathcal{N}(\Omega),$ where $\mathcal{E}^\infty_{M}(\Omega)$
consists of those nets $(f_{\varepsilon })_{\varepsilon}\in
\mathcal{E}(\Omega)^{(0,1]}$ with the property
\begin{equation}
\label{oeq}
(\forall \omega\subset\subset \Omega) (\exists a\in\mathbb{R})(\forall m\in \mathbb{N})(||f_{\varepsilon}||_{W^{m,\infty}(\omega)}=O(\varepsilon^{a})).
\end{equation}
Observe that $\mathcal{G}^\infty$ is a subsheaf of $\mathcal{G};$ it
has a similar role as $C^\infty$ in $\mathcal D'.$

\subsection{H\" older-Zygmund spaces}
We will employ the H\"{o}lder-Zygmund spaces \cite{hore,meyer1992,triebel2006}. We now collect some background material about these spaces. We start with H\"{o}lder spaces. Let $k\in\mathbb{N}_{0}$ and $\tau\in(0,1]$, then the global H\"
older space $\mathcal H^{k,\tau}(\mathbb R^d)$ \cite[Chap. 8]{hore} consists of those $C^{k}$ functions such that
\begin{equation} \label{held}
||f||_{\mathcal H^{k,\tau}(\mathbb R^d)}=||f||_{W^{k,\infty}(\mathbb R^d)}+\sup_{|\alpha|=k, x\neq y, x,y\in\mathbb R^d}
\frac{|f^{(\alpha)}(x)-f^{(\alpha)}(y)|}{|x-y|^\tau}<\infty.
\end{equation} 
The definition of the local space $ \mathcal H^{k,\tau}_{loc}(\Omega)$ is clear.

There are several ways to introduce the global Zygmund space $C^{r}_{\ast}(\mathbb{R}^{d})$ \cite{hore,meyer1992,triebel2006}. When $r=k+\tau,
k\in\mathbb N_0,  \tau\in (0,1)$, we have the equality $C^{r}_{\ast}(\mathbb{R}^{d})=\mathcal H^{k,\tau}(\mathbb R^d)$, but the Zygmund spaces are actually defined for all $r\in\mathbb{R}$. They are usually introduced via either a dyadic Littlewood-Paley resolution \cite{triebel2006} or a continuous Littlewood-Paley decomposition of the unity \cite{hore}. We follow the slightly more flexible approach from \cite{p-r-v} via generalized (continuous) Littlewood-Paley pairs (a dyadic version can be found in \cite[p. 7, Thrm. 1.7]{triebel2006}). Let $r\in\mathbb{R}$. We say that $\varphi,\psi\in\mathcal{S}(\mathbb{R}^{d})$ form a \emph{generalized Littlewood-Paley} pair (of order $r$) if they satisfy the following compatibility conditions:
\begin{equation}
\label{lpcond1}
(\exists \sigma>0,\eta\in(0,1))(|\hat{\varphi}(\xi)|>0  \mbox{ for }  \left|\xi\right|\leq \sigma \mbox{ and }|\hat {\psi}(\xi)|>0  \mbox{ for } \eta\sigma\leq\left|\xi\right|\leq \sigma)
\end{equation}
and 
\begin{equation}
\label{lpcond2}
 \int_{\mathbb{R}^{d}}t^{\alpha}\psi(t)dt=0  \: \mbox{ for } \: \left|\alpha\right|\leq[r].
\end{equation} 
When $r<0$, the vanishing requirement over the moments is dropped. 
Then, $C^{r}_{\ast}(\mathbb{R}^{d})$ is the space of all distributions $T\in\mathcal{S}'(\mathbb{R}^{d})$ satisfying:
\begin{equation}
\label{zeq} \left\|T\right\|_{C^{r}_{\ast}(\mathbb{R}^{d})}:=||T\ast\varphi||_{L^{\infty}(\mathbb{R}^{d})}
+\sup_{0<y\leq 1}y^{-r}||T\ast\psi_y||_{L^{\infty}(\mathbb{R}^{d})}<\infty.
\end{equation}
The definition and the norm (\ref{zeq}) (up to equivalence) are independent of the choice of the pair $(\varphi,\psi)$ 
as long as (\ref{lpcond1}) and (\ref{lpcond2}) hold \cite{p-r-v}. When $r=k+\tau,
k\in\mathbb N_0,  \tau\in (0,1)$, the norms (\ref{held}) and (\ref{zeq}) are equivalent.
A distribution $T\in\mathcal{D}'(\Omega)$ is said to belong to $C^{r}_{\ast,loc}(\Omega)$ if  for 
all $\rho\in\mathcal{D}(\Omega)$ we have $\rho T\in C^{r}_{\ast}(\mathbb{R}^{d})$.

\section{Classes of generalized functions with finite type regularities}
\label{classes}
In this paper we are interested in nets
 $(f_\varepsilon)_\varepsilon\in{\mathcal E}_{M}(\Omega)
$ 
such that for given $k\in \mathbb {N}$ there exists $s>0$ such that ($p\in[1,\infty]$)
\begin{equation}
\label{eqnetg1}
(\forall \omega\subset \subset \Omega)(||f_\varepsilon||_{W^{k,p}(\omega)}=O(\varepsilon^{
-s}),
\ \varepsilon \rightarrow 0).
\end{equation}
Observe that (\ref{eqnetg1}) is closely related to (\ref{oeq}). When $p=\infty$, such nets will be the representatives of, roughly speaking,
$C^{k-s}_{\ast,loc}-$generalized functions.

\begin{definition} \label{def1} Let $s\in\mathbb{R}$, $k\in\mathbb{N}_{0}$, and $p\in[1,\infty]$. 
\begin{itemize}
\item[(i)] A net $(f_{\varepsilon})_{\varepsilon}\in \mathcal{E}_{M}(\Omega)$ is said to belong to $\mathcal{E}^{k,-s}_{L^{p}_{loc},M}(\Omega)$ if (\ref{eqnetg1}) holds.
\item [(ii)] A generalized function $f=[(f_{\varepsilon})_{\varepsilon}]\in \mathcal{G}(\Omega)$ is said to belong to $\mathcal{G}^{k,-s}_{L^{p}_{loc}}(\Omega)$ if $(f_{\varepsilon})_{\varepsilon}\in \mathcal{E}^{k,-s}_{L^{p}_{loc},M}(\Omega)$. 
\item [(iii)] We set $\mathcal{G}^{\infty,-s}_{L^{p}_{loc}}(\Omega)=\bigcap_{k\in\mathbb{N}}\mathcal{G}^{k,-s}_{L^{p}_{loc}}(\Omega)$ and $\mathcal{E}^{\infty,-s}_{L^{p}_{loc}, M}(\Omega)=\bigcap_{k\in\mathbb{N}}\mathcal{E}^{k,-s}_{L^{p}_{loc}, M}(\Omega)$. 
\item [(iv)] When $p=\infty$, we simply write $\mathcal{G}^{k,-s}(\Omega)=\mathcal{G}^{k,-s}_{L^{\infty}_{loc}}(\Omega)$ and $\mathcal{E}^{k,-s}_{M}(\Omega)=\mathcal{E}^{k,-s}_{L^{\infty}_{loc},M}(\Omega)$.
\end{itemize}
\end{definition}

We list some properties of these classes of generalized functions in the next proposition. Their proofs follow immediately from Definition \ref{def1}.

\begin{proposition}
\label{propc1} Let $s\in\mathbb{R}$, $k\in\mathbb{N}_{0}\cup \left\{\infty\right\}$, and $p\in[1,\infty]$. 
\begin{itemize}
\item [(i)] ${\mathcal{G}}^{k,-s}_{L^{p}_{loc}}(\Omega)$ are vector spaces.
\item [(ii)] $
{\mathcal{G}}_{L^p_{loc}}^{k,-s}(\Omega)\subseteq
{\mathcal{G}}_{L^p_{loc}}^{k_1,-s_1}(\Omega)$ if $k\geq k_1$ and $s\leq
s_1$.
\item [(iii)] Let $P(D)$ be a
differential operator of order $m\leq k$ with constant coefficients. Then
$P(D):{\mathcal{G}}_{L^{p}_{loc}}^{k,-s}(\Omega)\rightarrow
{\mathcal{G}}_{L^{p}_{loc}}^{k-m,-s}(\Omega).$
\end{itemize}

\end{proposition}

The intuitive idea behind these notions is to measure the regularity of the net in terms of the two parameters $k$ and $s$: as the parameters $k$ increases and $s$ decreases, the net becomes more regular. Furthermore, it should be noticed that $f$ belongs to the algebra of regular generalized functions $\mathcal{G}^{\infty}(\Omega)$ if and only if $(\forall \omega\subset\subset \Omega)(\exists s)(f_{|\omega}\in \mathcal{G}^{\infty,-s}(\omega))$.

\section{Characterization of local regularity through association}
\label{cass}
In this section we characterize local regularity of distributions via either embedding in our classes $\mathcal{G}^{k,-s}_{L^{p}_{loc}}(\Omega)$ or association with its elements.  

Recall that we say that the net $(f_{\varepsilon})_{\varepsilon}\in\mathcal{E}(\Omega)^{(0,1]}$, or the generalized function $f=[(f_{\varepsilon})_{\varepsilon}]$, is (distributionally) associated to the distribution $T$ if $\lim_{\varepsilon\to0}f_{\varepsilon}=T$ in the weak topology of $\mathcal{D}'(\Omega)$, that is,
\begin{equation}
\label{eqnet1}
(\forall \rho \in \mathcal{D}(\Omega))(\langle T-f_\varepsilon, \rho
\rangle=o(1),\ \varepsilon\to0).
\end{equation}
We then write $
(f_{\varepsilon})_{\varepsilon}\sim T
$, or $f\sim T$. In many cases, the rate of approximation in (\ref{eqnet1}) may be much better than just $o(1)$; one can often profit  from the knowledge of such an additional useful asymptotic information. Let $R:(0,1)\to \mathbb{R}_{+}$ be a positive function such that $R(\varepsilon)=o(1),$ $\varepsilon\to0.$ We write $
T-f_{\varepsilon}=O(R(\varepsilon)) \ \mbox{  in } \mathcal{D}'(\Omega)
$
if 
\begin{equation*}
(\forall \rho \in \mathcal{D}(\Omega))(\langle T-f_\varepsilon, \rho
\rangle=O(R(\varepsilon)),\ \varepsilon\to0).
\end{equation*}

We begin with the following standard proposition. It gives the characterization of the embedding of $W^{k,p}_{loc}$. The case $p=\infty$ motivates our main results of this section. 

\begin{proposition}
\label{cprop2} Let $k\in\mathbb{N}_{0}$ and $p\in(1,\infty]$.
\begin{itemize}
\item [(a)] $\iota(W^{k,p}_{loc}(\Omega))=\iota(\mathcal{D}'(\Omega))\cap{\mathcal{G}}^{k,0}_{L^{p}_{loc}}(\Omega)$.
\item [(b)] More generally, if $(f_{\varepsilon})_{\varepsilon}\sim T\in\mathcal{D}'(\Omega)$ and $f=[(f_{\varepsilon})_{\varepsilon}]\in{\mathcal{G}}^{k,0}_{L^{p}_{loc}}(\Omega)$, then $T\in W^{k,p}_{loc}(\Omega)$. 
\end{itemize}
\end{proposition}

\begin{proof}
It is enough to show part b). We have that for every $|\alpha|\leq k$,
$((f^{(\alpha)}_\varepsilon)_{|\omega})_{\varepsilon}$ is weakly precompact in $L^{p}(\omega)$ if $p<\infty$, resp. weakly$^{\ast}$ precompact in $L^{\infty}(\omega)$. The rest follows from the distributional convergence of $f^{(\alpha)}_{\varepsilon}$ to $T^{(\alpha)}$. \end{proof}

We now formulate and prove the main results of this section. We focus on the case $p=\infty$ of the classes of generalized functions defined in Section \ref{classes}. 
\subsection{Characterization of $C^{r}_{\ast,loc}$ -- The role of $\mathcal{G}^{k,-s}$}
The next important theorem provides the precise characterization of those distributions that belong to $\mathcal{G}^{k,-s}(\Omega)$, they turn out to be elements of a Zygmund space. We only consider the case $s>0$; otherwise, one has $\mathcal{G}^{k,-s}(\Omega)\cap\iota(\mathcal{D}'(\Omega))=\left\{0\right\}$.

\begin{theorem}
\label{1mtheorem} Let $s>0$. We have $\mathcal{G}^{k,-s}(\Omega)\cap\iota(\mathcal{D}'(\Omega)) =\iota(C_{\ast,loc}^{k-s}(\Omega))$. 
\end{theorem}

Before giving the proof of Theorem \ref{1mtheorem}, we would like to discuss two corollaries of it. It is worth reformulating Theorem \ref{1mtheorem} in order to privilege the role of the Zygmund space. 

\begin{corollary}
\label{mcor1} Let $r\in\mathbb{R}$. If $k$ is any non-negative integer such that $k>r$, then $$\iota(C_{\ast,loc}^{r}(\Omega))=\mathcal{G}^{k,r-k}(\Omega)\cap\iota(\mathcal{D}'(\Omega)).$$
\end{corollary}

Corollary \ref{mcor1} can be used to give a striking proof of Oberguggenberger's regularity result \cite{ober001} for the algebra $\mathcal{G}^{\infty}(\Omega)$:

\begin{corollary}
\label{mcor2} We have $\iota(\mathcal{D}'(\Omega))\cap \mathcal{G}^{\infty}(\Omega)=\iota(C^{\infty}(\Omega))$.
\end{corollary}
\begin{proof} One inclusion is obvious. By localizing, it suffices to show that if $T\in\mathcal{E}'(\Omega)$ and $\iota(T)\in \mathcal{G}^{\infty,-s}(\Omega)$ for some $s\in\mathbb{R}_{+}$, then $T\in C^{\infty}(\Omega)$. Given any $k>0$, write $r=k-s$. Corollary \ref{mcor1} yields $T\in C_{*}^{k-s}(\Omega)$. Since this can be done for all $k$, we conclude that $f\in C^{\infty}(\Omega)$. 
\end{proof}

\begin
{proof}[Proof of Theorem \ref{1mtheorem}]
Observe that the statement of Theorem \ref{1mtheorem} is a local one. Thus, it is enough to show that 
$\mathcal{G}^{k,-s}(\Omega)\cap\iota(\mathcal{E}'(\Omega)) =\iota(C_{\ast}^{k-s}(\mathbb{R}^{d})\cap\mathcal{E}'(\Omega))$.
So, further on, in this proof we assume $T\in\mathcal E'(\Omega).$
Let us prove the reverse inclusion.  The partial derivatives continuously act on the Zygmund spaces \cite{hore} as $\partial^{m}: C_{\ast}^{\beta}(\mathbb{R}^{d})\mapsto C_{\ast}^{\beta-|m|}(\mathbb{R}^{d})$. Thus, if $T\in C_{\ast}^{k-s}(\mathbb{R}^{d})\cap\mathcal{E}'(\Omega)$ then $T^{(\alpha)}\in C_{\ast}^{-s}(\mathbb{R}^{d})\cap\mathcal{E}'(\Omega)$ for all $|\alpha|\leq k$. We can then apply \cite[Lemm. 5.2]{p-r-v} to each  $T^{(\alpha)}$ (with $\theta=\overline{\check{\phi}}$ in \cite[Lemm. 5.2, Eq. (5.7)]{p-r-v}) and conclude
$$||T^{(\alpha)}*\phi_\varepsilon||_{L^{\infty}(\mathbb{R}^{d})}\leq C\varepsilon^{-s} ||T^{(\alpha)}||_{C_*^{-s}(\mathbb{R}^{d})}.
$$
Thus $((T\ast \phi_{\varepsilon})_{|\Omega})_{\varepsilon}\in \mathcal{E}^{k,-s}_{M}(\Omega)$.

Assume now that $((T\ast\phi_{\varepsilon})_{|\Omega})_{\varepsilon}\in \mathcal{E}^{k,-s}_{M}(\Omega)$. 
We show that actually 
\begin{equation}
\label{eqextra1}
||T\ast \phi_{\varepsilon}||_{W^{k,\infty}(\mathbb{R}^{d})}=O(\varepsilon^{-s}), \ \ \ 0<\varepsilon\leq 1.
\end{equation}
 Indeed, let $\operatorname*{supp} T\subset\omega_{1}\subset \subset\omega_{2}\subset\subset \Omega$. It suffices to prove that for every multi-index $\alpha\in\mathbb{N}^{d}$
\begin{equation}
\label{eqextra}
\sup_{x\in\mathbb{R}^{d}\setminus \omega_{2} }\left|(T^{(\alpha)}\ast \phi_{\varepsilon})(x)\right|=O(1), \ \ \ 0<\varepsilon\leq 1.
\end{equation}
Let $A$ be the distance between $\overline{\omega}_{1}$ and  $\partial\omega_{2}$. Find $r$ such that  
$$
(\forall \rho\in\mathcal{E}(\mathbb{R}^{d}))(|\langle T^{(\alpha)}, \rho\rangle|<C\|\rho\|_{W^{r,\infty}(\omega_{1})}).
$$
Setting $\rho(\xi)=\phi_{\varepsilon}(x-\xi)$ and using the fact that $\phi$ is rapidly decreasing, we obtain,
$$
\sup_{x\in\mathbb{R}^{d}\setminus \omega_{2} }\left|(T^{(\alpha)}\ast \phi_{\varepsilon})(x)\right|< \tilde{C}\sup_{x\in\mathbb{R}^{d}\setminus \omega_{2} }\sup_{\xi\in \omega_{1}} (\varepsilon + \left|x-\xi\right|)^{-r-d}\leq \tilde{C}A^{-r-d},
$$
which yields (\ref{eqextra}). Next, set $g_{\varepsilon}=\varepsilon^{s}(T\ast \phi_{\varepsilon})$. Then, the growth estimate (\ref{eqextra1}) precisely tells us that $(g_{\varepsilon})_{\varepsilon}$ is a bounded net in the space $C_{b}^{k}(\mathbb{R}^{d})$, the Banach space of $k$-times continuously differentiable functions that are globally bounded together with all their partial derivatives of order $\leq k$. Since the inclusion mapping $C_{b}^{k}(\mathbb{R}^{d})\mapsto C^{k}_{\ast}(\mathbb{R}^{d})$ is obviously continuous, we obtain that $(g_{\varepsilon})_{\varepsilon}$ is a bounded net in the Zygmund space $C^{k}_{\ast}(\mathbb{R}^{d})$. Let $
\psi\in\mathcal{S}(\mathbb{R}^{d})$ be such that $(\phi,\psi)$ forms a generalized Littlewood-Paley pair of order $k$ (cf. (\ref{lpcond1}) and (\ref{lpcond2})). Then $T\ast\phi\in L^{\infty}(\mathbb{R}^{d})$ and
$$
\sup_{y\in(0,1]}y^{-k}||g_{\varepsilon}\ast\psi_{y}||_{L^{\infty}(\mathbb{R}^{d})}=\sup_{y\in(0,1]}\varepsilon^{s}y^{-k}||T\ast\phi_{\varepsilon}\ast\psi_{y}||_{L^{\infty}(\mathbb{R}^{d})}=O(1), \ \ \varepsilon\in(0,1].
$$
Setting $\varepsilon=y$ and $\psi_1= \phi\ast\psi$ in the previous estimate and noticing that $(\phi,\psi_{1})$ is again a Littlewood-Paley pair, we obtain
$$
\sup_{y\in(0,1]}y^{s-k}||T\ast(\psi_{1})_{y}||_{L^{\infty}(\mathbb{R}^{d})}<\infty, 
$$
which in turn implies that $T\in C^{k-s}_{\ast}(\mathbb{R}^{d}) $. 
\end{proof}

\subsection{Regularity via association} We now move to regularity analysis through association. Theorem \ref{1mtheorem} can be also used to recover the following general form of Corollary \ref{mcor2}, originally obtained in \cite{ps}.

\begin{theorem} \label{regas}
 Let $T\in\mathcal D'(\Omega)$ and let $f=[(f_\varepsilon)_\varepsilon]\in\mathcal{G}({\Omega})$ be associated to it. Assume that $f\in\mathcal{G}^{\infty}(\Omega)$. If 
$(f_{\varepsilon})_{\varepsilon}$ approximates $T$ with the following  convergence rate:
\begin{equation}
\label{req3}
(\exists b>0)(T-f_{\varepsilon}=O(\varepsilon^{b}) \mbox{ in }\mathcal{D}'(\Omega)),
\end{equation} 
then $T\in C^\infty(\Omega).$
\end{theorem}
\begin{proof}
Since the hypotheses and the conclusion of Theorem \ref{regas} are local statements, we may assume that $T\in\mathcal{E}'(\Omega)$ and there exists an open subset $\omega \subset \subset \Omega$ such that
\begin{equation} \label{cond}
\mathop{\rm supp} T, \mathop{\rm supp} f_\varepsilon \subset
\omega,\ \ \ \varepsilon\in(0,1].
\end{equation}
We will show that $T\in\mathcal{D}(\Omega)$. Our assumption now becomes $(f_\varepsilon)_\varepsilon\in \mathcal{E}_{M}^{\infty,-s}(\Omega)$ for some $s>0$. The support condition (\ref{cond}), the rate of convergence (\ref{req3}), and the equivalence between weak and strong boundedness on $\mathcal{E}'(\Omega)$ (Banach-Steinhaus theorem) yield
\begin{equation}
\label{conclusion1}(\exists r\in \mathbb{N})(\exists C>0)
(\forall \rho\in \mathcal{E}(\Omega))(\forall t\in(0,1])(
|\langle T-f_t , \rho \rangle|\leq Ct^{b}
\|\rho\|_{W^{r,\infty}(\omega)}).
\end{equation}
Let $\beta$ be an arbitrary positive number.  Then, by $(f_\varepsilon)_\varepsilon\in \mathcal{E}_{M}^{\infty,-s}(\Omega)$ and (\ref{conclusion1}), given any $k\in\mathbb{N}$, we can find positive constants $C_1$ and $C_2$ (depending only on $k,\phi$) such that
$$
||T\ast\phi_{\varepsilon}||_{W^{k,\infty}(\omega)}\leq C_{1}t^{-s}+C_{2}t^{b}\varepsilon^{-d-r-k}, \ \ \ \varepsilon,t\in(0,1].
$$
Find $\eta>0$ such that $\eta s/b<1/2$. Setting $t=\varepsilon^{k\eta/b}$, we obtain
$$
||T\ast\phi_{\varepsilon}||_{W^{k,\infty}(\omega)}\leq C_{1}\varepsilon^{-k/2}+C_{2}\varepsilon^{\eta k-d-r-k}, \ \ \ \varepsilon\in(0,1].
$$
We can now choose $k$ such that $\beta<\min \left\{k/2,\eta k-d-r\right\}$; the conclusion from the previous estimate is that $((T\ast \phi_{\varepsilon})_{|\omega})_{\varepsilon}\in \mathcal{E}_{M}^{k,\beta-k}(\omega)$, and hence, by Corollary \ref{mcor1}, $T\in C^{\beta}_{\ast,loc}(\omega)$. Since $\beta$ was arbitrary, it follows that $T\in C^{\infty}(\mathbb{R}^{d})$.
\end{proof}

We now discuss other sufficient criteria for regularity.
The ensuing result is directly motivated by Proposition \ref{cprop2}. We relax the growth constrains in it, and, by requesting an appropriate rate of convergence, we obtain two sufficient conditions for regularity of distributions.

 \begin{theorem}\label{mtheorem}
 Let $T\in\mathcal{D}'(\Omega)$ and let $f=[(f_\varepsilon)_\varepsilon]\in\mathcal{G}(\Omega)$ be
associated to it. Furthermore, let $k\in\mathbb{N}$. Assume that either of following pair of conditions hold:
\begin{enumerate}
\item[\textnormal{(i)}] $f\in \mathcal{G}^{k,-a}(\Omega)$, $\forall a>0$, namely,
\begin{equation}
\label{req2}
(\forall a>0)(\forall \omega\subset \subset \Omega)(\forall \alpha \in {\mathbb
N}^d,|\alpha|\leq k)(\sup_{x\in\omega}|f^{(\alpha)}_\varepsilon(x)|=O(\varepsilon^{-a})),
\end{equation}
and the convergence rate of  $(f_{\varepsilon})_{\varepsilon}$ to $T$ is as in $($\ref{req3}$)$.
 
\item [\textnormal{(ii)}] $f\in \mathcal{G}^{k,-s}(\Omega)$ for some $s>0$, and there is a rapidly decreasing function $R:(0,1]\to\mathbb{R}_{+}$, i.e., $(\forall a>0)(\lim_{\varepsilon\to0}\varepsilon^{-a}R(\varepsilon)=0)$, such that
\begin{equation}
\label{req4}
T-f_{\varepsilon}=O(R(\varepsilon))\  \mbox{ in }\mathcal{D}'(\Omega).
\end{equation}
\end{enumerate}
Then, $ T\in C_{*,\:loc}^{k-\eta}(\Omega)$ for every $\eta>0$. 

\end{theorem}

\begin{proof} By localization, it suffices again to assume that $T\in\mathcal{E}'(\Omega)$ and there exists an open subset $\omega \subset \subset \Omega$ such that (\ref{cond}) holds. The proof is analogous to that of Theorem \ref{regas}. 

(i) 
In view of the Banach-Steinhaus theorem, the conditions (\ref{req3}) and (\ref{cond}) imply (\ref{conclusion1}). Thus,
with $C_2=C||\phi||_{W^{r+k,\infty}(\mathbb{R}^{d})}$,
\begin{align*}
||T\ast \phi_{\varepsilon}||_{W^{k,\infty}(\omega)}&\leq  C_{2}t^b\varepsilon^{-d-r-k}+\|f_{t}\ast\phi_{\varepsilon}\|_{W^{k,\infty}(\omega)},
\\
&
\leq
C_{2}t^b\varepsilon^{-d-r-k}+\left\|\phi\right\|_{L^{1}(\mathbb{R}^{d})}\|f_{t}\|_{W^{k,\infty}(\omega)},
\
\ \ \ t,\varepsilon\in(0,1].
\end{align*}
By (\ref{req2}), given any $a>0$, there exists $M=M_{a}>0$ such that
$$
||T\ast \phi_{\varepsilon}||_{W^{k,\infty}(\omega)}\leq  C_{2}t^b\varepsilon^{-d-r-k}+M t^{-a},\ \ \ t,\varepsilon\in(0,1].
$$
By taking $t=\varepsilon^{(k+r+d)/b}$, it follows that
$$||T\ast \phi_{\varepsilon}||_{W^{k,\infty}(\omega)}\leq  C_{2}+M \varepsilon^{-a(k+r+d)/b},\ \ \ \varepsilon\in(0,1].
$$
If we take sufficiently small $a$, we conclude that $(T\ast \phi_{\varepsilon})_{\varepsilon}\in \mathcal{E}_{M}^{k,-\eta}(\omega)$ for all $\eta>0$, and the assertion follows from Theorem \ref{1mtheorem}.

(ii) The relation (\ref{req4}), the fact that $R$ is rapidly decreasing, and the Banach-Steinhaus theorem imply 
$$
(\exists r \in \mathbb{N})(\forall a>0) (\exists C>0)(\forall \rho \in
\mathcal{E}(\Omega)) (\forall t\in(0,1])(|\langle T-f_t, \rho
\rangle|\leq C t^{a}||\rho||_{W^{r,\infty}(\omega)}).
$$
As in part (i),
we have
$$||T\ast \phi_{\varepsilon}||_{W^{k,\infty}(\omega)}\leq  Ct^a\varepsilon^{-d-r-k}+ \left\|\phi\right\|_{L^{1}(\mathbb{R}^{d})}\|f_{t}\|_{W^{k,\infty}(\omega)},\ \ \ t,\varepsilon\in(0,1],
$$
for some constant $C=C_{a}$. Since
$(f_\varepsilon)_\varepsilon\in \mathcal{E}^{k,-s}_{M}$, there is another constant $C=C_{s,a}>0$ such that
$$
||T\ast \phi_{\varepsilon}||_{W^{k,\infty}(\omega)}\leq  Ct^a\varepsilon^{-d-r-k}+Ct^{-s},\ \ \ t,\varepsilon\in(0,1].
$$
Setting $t=\varepsilon^{(k+r+d)/a}$, we have
\begin{equation*} 
||T\ast \phi_{\varepsilon}||_{W^{k,\infty}(\omega)}\leq  C+C\varepsilon^{-s(k+r+d)/a},\ \ \ \varepsilon\in(0,1].
\end{equation*}
Thus, taking large enough $a>0$, one establishes $\iota(T)\in\mathcal{G}^{k,-\eta}(\omega)$ for all $\eta>0$. The conclusion $T\in C_{\ast}^{k-\eta}(\mathbb{R}^{d})$ follows once again from Theorem \ref{1mtheorem}.
\end{proof}

The hypotheses (\ref{req3}) and (\ref{req4}) are  essential parts of (i) and (ii) in Theorem \ref{mtheorem}; we illustrate that fact in the next two examples.

\begin{example}
\label{rex3} Consider the generalized function $f= [(\left|\log \varepsilon \right|^{d}\phi(\:\cdot\:\left|\log \varepsilon\right|))_{\varepsilon}]$. Clearly, $f\in \mathcal{G}^{\infty,-a}(\mathbb{R}^{d})$, $\forall a>0$. Moreover, $f\sim\delta$, the Dirac delta distribution. The conclusion of Theorem \ref{mtheorem} fails in this example because the rate of convergence is too slow. 
\end{example}

\begin{example}
\label{rex4}
Let $T\in\mathcal{E}'(\Omega)$ and $k>2s>0$. Suppose that $T\in C^{k-2s}_{\ast}(\mathbb{R}^{d})$ but $T\notin C^{k-s}_{\ast}(\mathbb{R}^{d})$. By Theorem \ref{1mtheorem}, $\iota(T)\in \mathcal{G}^{k,-2s}(\mathbb{R}^{d})$. However, the conclusion of Theorem \ref{mtheorem} fails for $T$ because the approximation rate is actually much slower than (\ref{req4}).
\end{example}

For distributions $T\in\mathcal{E}'(\Omega)$, part (i) of Theorem \ref{mtheorem} is applicable to the regularization net $f_{\varepsilon}=(T\ast \phi_{\varepsilon})_{|\Omega}$; however, for this particular case Theorem \ref{1mtheorem} provides the same regularity conclusion.

\section{Global Zygmund-type spaces and algebras}
\label{szh}
Let $r\in {\mathbb R},$ H\" {o}rmann (\cite{her}) defined the Zygmund-type space of generalized functions 
$\tilde{\mathcal{G}}_*^r(\mathbb{R}^d)$ via representatives
$(u_\varepsilon)_\varepsilon$ satisfying, for each $\alpha \in {\mathbb
N}_{0}^{d},$
\begin{equation}\label{imbhz}
\|u^{(\alpha)}_{\varepsilon}\|_{L^{\infty}(\mathbb{R}^{d})}=\begin{cases} {O}(1),& 0\leq|\alpha|<r, \\
{O}(\log(1/\varepsilon)), & |\alpha|=r\in {\mathbb N}_{0} \\
{O}(\varepsilon^{r-|\alpha|}),& |\alpha|>r.
\end{cases}
\ \ \ \mbox{        as } \varepsilon \rightarrow 0,
\end{equation}

We shall propose in this section several other Zygmund-type classes of generalized functions. Since we are interested in global properties, it appears that the most natural framework to define them  is the algebra $\mathcal{G}_{L^{\infty}}(\mathbb{R}^{d})$, defined below in Subsection \ref{linfty}, and not the usual Colombeau algebra $\mathcal{G}(\mathbb{R}^{d})$. Otherwise, the definitions would depend on representatives, and more seriously, some global properties that are intrinsically encoded in such spaces would be totally lost. Therefore, we have decided to study first $\mathcal{G}_{L^{\infty}}(\mathbb{R}^{d})$. Subsection \ref{zyg} is devoted to Zygmund-type classes of generalized functions and global regularity results. In Subsection \ref{holderalgebra}, we introduce H\" {o}lder-type classes of generalized functions.

\subsection{The algebra $\mathcal{G}_{L^{\infty}}(\mathbb{R}^{d})$}
\label{linfty}
The globally $L^{\infty}$-based algebra of generalized functions is defined as follows (cf.  \cite{ober001,nop}). First consider the algebra
$$\mathcal{E}_{L^{\infty},M}(\mathbb{R}^{d})=\left\{(u_{\varepsilon})_{\varepsilon}\in \mathcal{E}_{M}(\mathbb{R}^{d});\ (\forall \alpha\in\mathbb{N}_0)(\exists a\in\mathbb{R})
(||u^{(\alpha)}_{\varepsilon}||_{L^{\infty}(\mathbb{R}^{d})}=O(\varepsilon^{a}))\right\}$$
and the ideal
$$\mathcal{N}_{L^{\infty}}(\mathbb{R}^{d})=\left\{(u_{\varepsilon})_{\varepsilon}\in \mathcal{E}_{M}(\mathbb{R}^{d});\ (\forall \alpha\in\mathbb{N}_0)(\forall b\in\mathbb{R})
(||u^{(\alpha)}_{\varepsilon}||_{L^{\infty}(\mathbb{R}^{d})}=O(\varepsilon^{b}))\right\}$$
The algebra $\mathcal{G}_{L^{\infty}}(\mathbb{R}^{d})$ is defined as the quotient  $$\mathcal{G}_{L^{\infty}}(\Omega)=\mathcal{E}_{L^{\infty},M}(\mathbb{R}^{d})/\mathcal{N}_{L^{\infty}}(\mathbb{R}^{d}).$$
The natural class of distributions that can be embedded into $\mathcal{G}_{L^{\infty}}(\mathbb{R}^{d})$ is the Schwartz distribution space of the so-called bounded distributions \cite{sch}. More precisely, this space is given by
$$
\mathcal{D}'_{L^{\infty}}(\mathbb{R}^{d})=\bigcup_{m\in\mathbb{N}_{0}}W^{-m,\infty}(\mathbb{R}^{d})=\bigcup_{s\in \mathbb{R}}C^{s}_{\ast}(\mathbb{R}^{d}).
$$
It is the dual of the test function space \cite{sch} $\mathcal{D}_{L^{1}}(\mathbb{R}^{d})=\bigcap_{m\in\mathbb{N}}W^{m,1}(\mathbb{R}^{d})$.
Clearly, $\iota:\mathcal{D}'_{L^{\infty}}(\mathbb{R}^{d})\mapsto \mathcal{G}_{L^{\infty}}(\mathbb{R}^{d})$ given as usual by $\iota(T)=[(T\ast\phi_{\varepsilon})_{\varepsilon}]$ provides a natural embedding. On the other hand, the embedding does not extend to $\mathcal{S}'(\mathbb{R}^{d})$, and more interestingly, as long as $\iota(T)\in \mathcal{G}_{L^{\infty}}(\mathbb{R}^{d})$ for a tempered distribution, it is forced to belong to $\mathcal{D}'_{L^{\infty}}(\mathbb{R}^{d})$. 

\begin{theorem}
\label{theoremLinfinity} Let $T\in\mathcal{S}'(\mathbb{R}^{d})$. If $(T\ast\phi_{\varepsilon})_{\varepsilon}\in \mathcal{E}_{L^{\infty},M}(\mathbb{R}^{d})$, then $T\in\mathcal{D}'_{L^{\infty}}(\mathbb{R}^{d})$.
\end{theorem}   
\begin{proof}
Because of Schwartz characterization \cite[Chap. VI]{sch} of $\mathcal{D}'_{L^{\infty}}(\mathbb{R}^{d})$, it would be enough to show that, for each $\rho\in\mathcal{S}(\mathbb{R}^{d})$, $T\ast \rho \in C_{b}(\mathbb{R}^{d})$, the Banach space of continuous and bounded functions. In order to show so, we will use the vector-valued Tauberian theory for class estimates developed in \cite[Sect. 7]{pv} (see also \cite{DZ,vipira}). Define the vector-valued distribution $\mathbf{T}$ whose action on test functions $\rho\in\mathcal{S}(\mathbb{R}^{d})$ is given by $\left\langle \mathbf{T},\rho\right\rangle=T\ast \check{\rho}$. Therefore, we must show that $\mathbf{T}\in\mathcal{S}'(\mathbb{R}^{d},C_{b}(\mathbb{R}^{d}))$. Since $T$ is tempered, there exists $N\in\mathbb{N}$ such that $\mathbf{T}$ takes values in the Banach space $X$ consisting of continuous functions $g$ on $\mathbb{R}^{d}$ such that $\left\|g\right\|_{X}:=\sup_{t\in\mathbb{R}^{d}}(1+\left|t\right|)^{-N}\left|g(t)\right|<\infty$. Clearly, the inclusion mapping $C_{b}(\mathbb{R}^{d})\mapsto X$ is continuous. On the other hand, we have the local class estimate 
$$\left\|(\mathbf{T}\ast\phi_{\varepsilon})(x)\right\|_{L^{\infty}(\mathbb{R}^{d})}=\sup_{\xi\in\mathbb{R}^{d}}\left|(T\ast\phi_{\varepsilon})(x+\xi)\right|=\left\|T\ast\phi_{\varepsilon}\right\|_{L^{\infty}(\mathbb{R}^{d})}=O(\varepsilon^{-a}),$$
for some $a>0$. Thus, in view of the \cite[Thm. 7.9]{pv}, we obtain the desired conclusion $\mathbf{T}\in\mathcal{S}'(\mathbb{R}^{d},C_{b}(\mathbb{R}^{d}))$.
\end{proof}

Let us note that $\mathcal{E}_{L^{\infty},M}(\mathbb{R}^{d})\subset\mathcal{E}_{M}(\mathbb{R}^{d})$ is a differential subalgebra and $\mathcal{N}_{L^{\infty}}(\mathbb{R}^{d})\subset \mathcal{N}(\mathbb{R}^{d})$. There is a canonical differential algebra mapping $\mathcal{G}_{L^{\infty}}(\mathbb{R}^{d})\rightarrow\mathcal{G}(\mathbb{R}^{d})$; however, this mapping is not injective. Hence
$\mathcal{G}_{L^{\infty}}(\mathbb{R}^{d})$ cannot be seen as a differential subalgebra of $\mathcal{G}(\mathbb{R}^{d})$.

\begin{example} This example shows that the canonical mapping $\mathcal{G}_{L^{\infty}}(\mathbb{R}^{d})\rightarrow\mathcal{G}(\mathbb{R}^{d})$ is not injective. Equivalently, we find a net $(u_{\varepsilon})_{\varepsilon}\in \mathcal{N}(\mathbb{R}^{d})\cap\mathcal{E}_{L^{\infty},M}({\mathbb{R}^{d}})$ which does not belong to $\mathcal{N}_{L^{\infty}}(\mathbb{R}^{d})$. Let $\rho\in\mathcal{D}(\mathbb{R}^{d})$ be non-trivial and supported by the ball with center at the origin and radius $1/2$. Consider the net of smooth functions
$$
u_{\varepsilon}(x)=\sum_{n=0}^{\infty} \frac{\chi_{[(n+1)^{-1},1]}(\varepsilon)}{(n+1)^{2}}\rho(x-2ne_{1}),
$$
where $\chi_{[(n+1)^{-1},1]}$ is the characteristic function of the interval $[1/(n+1),1]$ and $e_{1}=(1,0,\dots,0)$. Then, clearly $(u_{\varepsilon})_{\varepsilon}\in \mathcal{N}(\mathbb{R}^{d})$ because on compact sets it identically vanishes for small enough $\varepsilon$. On the other hand, by the elementary asymptotic formula $\sum_{n<x}n^{-2}=\zeta(2)-x^{-1}+O(x^{-2})$ (see, e.g., \cite[p. 32]{estrada-kanwal}), we have
$$
||u_{\varepsilon}||_{W^{m,\infty}}=||\rho||_{W^{m,\infty}}\sum_{\frac{1}{\varepsilon}-1\leq n}^{\infty} \frac{1}{(n+1)^{2}}= \varepsilon\: ||\rho||_{W^{m,\infty}}+O(\varepsilon^{2}), \ \ \ \varepsilon\to0.
$$
Thus, the net satisfies all the requirements.
\end{example}

\subsection{Global Zygmund classes}
\label{zyg}
We come back to H\" {o}rmann's Zygmund class of generalized functions. We slightly modify his definition. Given $r\in {\mathbb R}$, define  
$\tilde{\mathcal{G}}_*^r(\mathbb{R}^d)$ as 
the space of those $u=[(u_{\varepsilon})_{\varepsilon}]\in\mathcal{G}_{L^{\infty}}(\mathbb{R}^{d})$ such that $(u_{\varepsilon})_{\varepsilon}$ satisfies (\ref{imbhz}). Originally \cite{her}, the ``tilde'' did not appear in the notation but since we will introduce
a new  definition, which is intrinsically related to the classical definition of Zygmund spaces,
we leave the notation $\mathcal{G}_*^r(\mathbb{R}^d)$ for our space.

\begin{definition} \label{zig} Let $r\in\mathbb{R}$ and let $\varphi,\psi\in \mathcal{S}'(\mathbb{R}^{d})$ be a  pair satisfying (\ref{lpcond1}) and (\ref{lpcond2}) (i.e., a generalized Littlewood-Paley pair).
The space
$\mathcal{G}_*^{r}(\mathbb{R}^d)=\mathcal{G}_*^{r,0}(\mathbb{R}^d), $ called  the Zygmund space of
generalized functions of $0-$growth order,
 consists of  those $u=[(u_{\varepsilon})_{\varepsilon}]\in\mathcal{G}_{L^{\infty}}(\mathbb{R}^d)$
 such that 
\begin{equation}\label{secz}
||u_\varepsilon||_{C^{r}_{\ast}(\mathbb{R}^{d})}=||u_\varepsilon*\varphi||_{L^\infty(\mathbb R^d)}+\sup_{0<y\leq 1}y^{-r}||u_\varepsilon*\psi_y||_{L^\infty(\mathbb R^d)}=O(1).
\end{equation}
Moreover,
$\mathcal{G}_*^{r,-s}(\mathbb{R}^d), $ the Zygmund space of
generalized functions of $-s$-growth order, 
consists of those $u=[(u_{\varepsilon})_{\varepsilon}]\in\mathcal{G}_{L^{\infty}}(\mathbb{R}^d)$ such that
$[(\varepsilon^{s}u_\varepsilon)_\varepsilon]\in \mathcal{G}_*^{r}(\mathbb{R}^d)$.
\end{definition}

Observe that Definition \ref{zig} is independent of the choice of representatives. The main properties of these spaces are summarized in the next theorem. In particular, we show the embedding of 
the ordinary Zygmund spaces of functions and characterize those distributions which, 
after embedding, belong to our generalized Zygmund classes. 
\begin{theorem}\label{t-her}
The following properties hold,
\begin{itemize}
\item [(i)]  
 $\iota(C^{r}_{\ast}(\mathbb{R}^{d}))=\mathcal{G}^{r}_{\ast}(\mathbb{R}^{d})\cap \iota(\mathcal{D}'_{L^{\infty}}(\mathbb{R}^{d}))$.
 \item [(ii)] $\mathcal{G}^{r,-s}_{\ast}(\mathbb{R}^{d})\cap \iota(\mathcal{D}'_{L^{\infty}}(\mathbb{R}^{d}))\subset\iota(C^{r-s}_{\ast}(\mathbb{R}^{d}))$.
\item [(iii)] $\mathcal{G}_*^{r_1,-s}(\mathbb{R}^d)
\subset\mathcal{G}_*^{r,-s}(\mathbb{R}^d)$ if $r_1> r;$
$P(D)\mathcal{G}_*^{r,-s}(\mathbb{R}^d)\subset
\mathcal{G}_*^{r-m,-s}(\mathbb{R}^d),$ where $P(D)$ is a differential operator with constant coefficients and order $m$.
\item [(iv)] If $r_{1}+r_{2}>0$, then $$\mathcal{G}_*^{r_1,-s_1}(\mathbb{R}^d)\cdot\mathcal{G}_*^{r_{2},-s_2}(\mathbb{R}^d)\subset\mathcal{G}_*^{p,-s_1-s_2}(\mathbb{R}^d), \ \ p=\min\left\{r_1,r_2\right\}.$$
In particular, $\mathcal{G}_*^{r,-s}(\mathbb{R}^{d})$ is an algebra if $s=0$ and $r>0$.
\end{itemize}
\end{theorem}
\begin{proof}
 (i) and (ii). We first show that $\iota(C^{r}_{\ast}(\mathbb{R}^{d}))\subset\mathcal{G}^{r}_{\ast}(\mathbb{R}^{d})$. Let $u\in C^r_*(\mathbb R^d)$ and $u_{\varepsilon}=u*\phi_\varepsilon$. Obviously,
$$||u_\varepsilon*\varphi||_{L^{\infty}(\mathbb{R}^{d})}\leq \left\|u\ast \varphi\right\|_{L^{\infty}(\mathbb{R}^{d})}\left\|\phi\right\|_{L^{1}(\mathbb{R}^{d})}$$
and
$$\sup_{0<y\leq 1}y^{-r}||u_{\varepsilon}\ast\psi_{y}||_{L^{\infty}(\mathbb{R}^{d})}\leq ||\phi||_{L^{1}(\mathbb{R}^{d})}\sup_{0<y\leq 1}y^{-r}||u\ast \psi_{y}||_{L^{\infty}(\mathbb{R}^{d})}.$$
Let us now prove the inclusion $\mathcal{G}^{r,-s}_{\ast}(\mathbb{R}^{d})\cap \iota(\mathcal{D}'_{L^{\infty}}(\mathbb{R}^{d}))\subset\iota(C^{r-s}_{\ast}(\mathbb{R}^{d}))$. The proof is similar to the last part of the proof of Theorem \ref{1mtheorem}. So, let $\iota(u)=[(u\ast\phi_{\varepsilon})_{\varepsilon}]\in \mathcal{G}^{r,-s}_{\ast}(\mathbb{R}^{d})$, where $u\in\mathcal{D}'_{L^{\infty}}(\mathbb{R}^{d})$.  We have freedom of choice for the Littlewood-Paley pair in (\ref{secz}). Let then $
\psi\in\mathcal{S}(\mathbb{R}^{d})$ be such that $(\phi,\psi)$ forms a generalized Littlewood-Paley pair of order $\max\left\{r,r-s\right\}$ (cf. (\ref{lpcond1}) and (\ref{lpcond2})). We have $u\ast\phi\in L^{\infty}(\mathbb{R}^{d})$; on the other hand, setting $\varepsilon=y\leq 1$ and observing that $(\phi*\psi)_{y}=\phi_y*\psi_y,$ one obtains that 

$$ \sup_{0<y\leq 1}y^{s-r}||u*(\phi*\psi)_{y}||_{L^{\infty}(\mathbb{R}^{d})}=\sup_{0<y\leq 1}y^{s-r}||u*\phi_y*\psi_y||_{L^{\infty}(\mathbb{R}^{d})} <
\infty
.$$
Noticing that $(\phi,\phi\ast\psi)$ is again a generalized Littlewood-Paley pair of order $r-s$, we conclude $T\in C^{r-s}_{\ast}(\mathbb{R}^{d}) $. 

(iii) The first part is clear. The second part follows from the fact \cite{hore} that $P(D)$ continuously maps the classical Zygmund space $C^{r}_{\ast}(\mathbb{R}^{d})$ into $C^{r-m}_{\ast}(\mathbb{R}^{d})$.

(iv) It is a consequence of \cite[Prop. 8.6.8]{hore}.
Actually, we have, by this proposition,
that there exists $\varepsilon_0\in(0,1]$ and $K=K(r_1,r_2)$, which does not depend on $\varepsilon$, such that
$$||\varepsilon^{s_1+s_2}u_{1,\varepsilon}u_{2,\varepsilon}||_{C^{p}_{\ast}(\mathbb{R}^{d})} \leq K
||\varepsilon^{s_1}u_{1,\varepsilon}||_{C^{r_1}_{\ast}(\mathbb{R}^{d})}||\varepsilon^{s_2}u_{2,\varepsilon}||_{C^{r_2}_{\ast}(\mathbb{R}^{d})},\ \varepsilon\leq \varepsilon_{0}.
$$
\end{proof}
\begin{remark}
 As in the case of multiplication of continuous functions,  we  have that $[((u_1u_2)*\phi_\varepsilon)_\varepsilon]\neq
[(u_1*\phi_\varepsilon)_\varepsilon][(u_2*\phi_\varepsilon)_\varepsilon]$  but these products are  associated.
\end{remark}

In analogy with Definition \ref{def1}, we can also introduce some other classes of generalized functions. They are now closely related to the classical global Zygmund spaces.

\begin{definition} \label{def3} Let $s\in\mathbb{R}$ and $k\in\mathbb{N}_{0}$. 
\begin{itemize}
\item[(i)] A net $(f_{\varepsilon})_{\varepsilon}\in \mathcal{E}_{L^{\infty},M}(\mathbb{R}^{d})$ is said to belong to $\mathcal{E}^{k,-s}_{L^{\infty},M}(\mathbb{R}^{d})$ if 
$$
\left\|f_{\varepsilon}\right\|_{W^{k,\infty}(\mathbb{R}^{d})}=O(\varepsilon^{-s}).
$$

\item [(ii)] A generalized function $f=[(f_{\varepsilon})_{\varepsilon}]\in \mathcal{G}_{L^{\infty}}(\mathbb{R}^{d})$ is said to belong to $\mathcal{G}^{k,-s}_{L^{\infty}}(\mathbb{R}^{d})$ if $(f_{\varepsilon})_{\varepsilon}\in \mathcal{E}^{k,-s}_{L^{\infty},M}(\mathbb{R}^{d})$. 
\end{itemize} 
\end{definition}

Definition \ref{def3} does not depend on the choice of representatives. The next theorem characterizes those distributions that belong to $\mathcal{G}_{L^{\infty}}(\mathbb{R}^{d})$ and gives an inclusion relation for H\" {o}rmann class $\tilde{\mathcal{G}}_{\ast}^{r}(\mathbb{R}^{d})$. Let us mention that a version of Proposition \ref{propc1} also holds for $\mathcal{G}^{k,-s}_{L^{\infty}}(\mathbb{R}^{d})$. 

\begin{theorem}
\label{theoremextra} Let $r\in\mathbb{R}$ and $s>0$.
\begin{itemize}
\item [(i)] We have $\mathcal{G}^{k,-s}_{L^{\infty}}(\mathbb{R}^{d})\cap\iota(\mathcal{D}'_{L^{\infty}}(\Omega)) =\iota(C_{\ast}^{k-s}(\mathbb{R}^{d}))$. 
\item [(ii)] Given any integer $k>r$, we have $\iota(C_{\ast}^{r}(\mathbb{R}^{d}))=\mathcal{G}^{k,r-k}_{L^{\infty}}(\mathbb{R}^{d})\cap \iota(\mathcal{D}'_{L^{\infty}}(\Omega))$.
\item [(iii)] There holds 
$$\tilde{\mathcal{G}}_{\ast}^{r}(\mathbb{R}^{d})\subset \bigcap_{k>r}\mathcal{G}^{k,r-k}_{L^{\infty}}(\mathbb{R}^{d}).$$
\end{itemize}
\end{theorem}
\begin{proof}
The property (iii) follows directly from the definitions. Observe that (i) and (ii) are equivalent. On the other hand, a straightforward modification of the proof of Theorem \ref{1mtheorem} yields (i), we leave the details of such a modification to the reader.
\end{proof}

The following remarks make some partial comparisons between our definition and H\"ormann's definition \cite{her}. We also formulate an open question.
\begin{remark}\label{adrem}
Clearly, if  $u=[(u_\varepsilon)_\varepsilon]\in\tilde{\mathcal{G}}_*^r(\mathbb{R}^{d})$, then $[(u_\varepsilon*\phi_\varepsilon)_\varepsilon]\in\tilde{\mathcal{G}}_*^r(\mathbb{R}^{d})$
but the opposite does not hold, in general. 
However, $u=[(u_\varepsilon)_\varepsilon]$ and $u=[(u_\varepsilon*\phi_\varepsilon)_\varepsilon]$
are equal in the sense of generalized distributions, which means that 
$$\langle u_\varepsilon*\phi_\varepsilon-u_\varepsilon,\theta \rangle = o(\varepsilon^p) \;\mbox{ for every }  p \; \mbox{ and every } \;  \theta\in{\mathcal D}(\mathbb R^d).$$  
\end{remark}
\begin{remark} Let $u=[(u_{\varepsilon})_{\varepsilon}]\in \mathcal{G}^{r}_{\ast}(\mathbb{R}^{d})$. We show that  $[(u_{\varepsilon}\ast\phi_{\varepsilon})_{\varepsilon}]\in\tilde{\mathcal{G}}^{r}_{\ast}(\mathbb{R}^{d})$. For this, we will make use of Lemma 8.6.5 of \cite{hore}, which asserts that given $\kappa\in\mathcal{S}(\mathbb{R}^{d})$, there exist constants $K_{r,\alpha}$, $\alpha\in\mathbb{N}_{0}$, such that for all $v\in C_{\ast}^{r}(\mathbb{R}^{d})$ and $0<y\leq 1$, the following estimate holds with as usual  $\kappa_{y}=y^{-d}\kappa(\:\cdot\:/y)$,
\begin{equation}\label{imbhzu}
\|(v\ast\kappa_{y})^{(\alpha)}\|_{L^{\infty}(\mathbb{R}^{d})}
\leq\begin{cases} K_{r,\alpha}||v||_{C^{r}_{*}(\mathbb{R}^{d})},& 0\leq|\alpha|<r, \\
K_{r,\alpha}||v||_{C^{r}_{*}(\mathbb{R}^{d})}(1+\log(1/y)), & |\alpha|=r\in {\mathbb N}_{0}, \\
K_{r,\alpha}||v||_{C^{r}_{*}(\mathbb{R}^{d})}(y^{r-|\alpha|}),& |\alpha|>r.
\end{cases}
 \end{equation}
Thus, if we employ (\ref{imbhzu}) with $v=u_{\varepsilon}$, $\kappa=\phi$, and $y=\varepsilon$, together with the the fact that $||u_\varepsilon||_{C^{r}_{*}(\mathbb{R}^{d})}$ is uniformly bounded with respect to $\varepsilon,$ we obtain at once $[(u_{\varepsilon}\ast\phi_{\varepsilon})_{\varepsilon}]\in\tilde{\mathcal{G}}^{r}_{\ast}(\mathbb{R}^{d})$, as claimed. At this point we should mention that the precise relation between the spaces $\mathcal{G}^{r}_{\ast}(\mathbb{R}^{d})$ and $\tilde{\mathcal{G}}^{r}_{\ast}(\mathbb{R}^{d})$ is still unknown; therefore, we can formulate an \emph{open question: find the precise inclusion relation between these two spaces.}
\end{remark}
\begin{remark}
 As seen from the given assertions, our Zygmund generalized function spaces are suitable for the analysis of pseudodifferential operators.
\end{remark}

\subsection{H\" older-type spaces and algebras of generalized functions.}
\label{holderalgebra} We end this article by dealing with H\" olderian-type classes of generalized functions. We will employ the norm (\ref{held}).

\begin{definition}\label{helnp}
Let $k\in\mathbb N_0,$ $s\in\mathbb{R},$ $\tau\in(0,1]$ and
let $u=[(u_{\varepsilon})_{\varepsilon}]\in \mathcal G_{L^{\infty}}(\mathbb R^d)$. It is said that $u\in
{\mathcal G}^{k,\tau,-s}_{L^{\infty}}(\mathbb{R}^d)$ if 
\begin{equation} \label{hel0}
||u_\varepsilon||_{\mathcal{H}^{k,\tau}(\mathbb{R}^{d})}= O(\varepsilon^{-s}).
\end{equation}
\end{definition}

Recall \cite{hore} the classical situation. Let $k\in \mathbb N_0 $, then
$\mathcal{H}^{k,1}(\mathbb{R}^d)\subsetneqq C_*^{k+1}(\mathbb R^d)$; but if $\tau\in(0,1)$, then $\mathcal{H}^{k,\tau}(\mathbb{R}^d)= C_*^{k+\tau}(\mathbb R^d)$. In our context, we have,

\begin{proposition}
\label{cprop3}
If $r=k+\tau, \tau\in (0,1),$ then
$\mathcal{G}_*^{r,s}(\mathbb{R}^d)={\mathcal
G}^{k,\tau,s}(\mathbb{R}^d).$
\begin{proof} There exists $C>0$ such that for every $\varepsilon\leq 1,$
$$
C^{-1} ||\varepsilon^su_\varepsilon||_{C^{k+\tau}_{\ast}(\mathbb{R}^{d})}\leq ||\varepsilon^su_\varepsilon||_{\mathcal H^{k,\tau}(\mathbb{R}^{d})}
\leq C ||\varepsilon^su_\varepsilon||_{C^{k+\tau}_{\ast}(\mathbb{R}^{d})},
 $$
as follows from the equivalence between the norms (\ref{held}) and (\ref{zeq}). This implies the assertion.
\end{proof}
\end{proposition}

Because of Proposition \ref{cprop3}, we will consider
below only the case $\mathcal{G}^{k,1,s}(\mathbb{R}^d)$.
\begin{proposition}\label{p-cnp} Let $k\in\mathbb{N}_{0}$ and $s\in\mathbb{R}$.
\begin{itemize}
\item [(i)] ${\iota}(\mathcal{H}^{k,1}(\mathbb{R}^d))=\mathcal{G}^{k,1,0}(\mathbb{R}^d)\cap \iota(\mathcal{D}'_{L^{\infty}}(\mathbb{R}^{d})).$
\item [(ii)] $\mathcal{G}^{k,1,s}(\mathbb{R}^d)\subsetneqq
\mathcal G_*^{k+1,s}(\mathbb R^d).$
\item [(iii)] $\mathcal{G}^{k_1,1,s}(\mathbb{R}^d)
\subset\mathcal{G}^{k,1,s}(\mathbb{R}^d)$ if $k_1> k.$
\item [(iv)] Let $P(D)$ be a
differential operator of order $m<k$ with constant
coefficients. Then
$P(D):\mathcal{G}^{k,\tau,s}(\mathbb R^d)\rightarrow
\mathcal{G}^{k-m,\tau,s}(\mathbb R^d)$.
\item [(v)] Concerning the multiplication, we have
 $$\mathcal{G}^{k_1,1,s}(\mathbb{R}^d)\cdot\mathcal{G}^{k_2,1,s}(\mathbb{R}^d)
 \subset\mathcal{G}^{p,1,2s}(\mathbb{R}^d),$$
where  $p=\min\{k_1,k_2\}$. In particular,
$\mathcal{G}^{k_1,1,s}(\mathbb{R}^d)$ is an algebra if and only if
$s=0.$
\end{itemize}
\end{proposition}
\begin{proof} The proofs of the assertions (ii), (iii), (iv) and (v) are clear. 
We will prove (i). The direct inclusion follows from the definition. Suppose that $T\in\mathcal{D}'_{L^{\infty}}(\mathbb{R}^{d})$ is such that $[(T_{\varepsilon})_{\varepsilon}]\in \mathcal{G}^{k,1,0}(\mathbb{R}^d)$ where $T_{\varepsilon}=T\ast\phi_{\varepsilon}$. By  assumption $\{T_\varepsilon^{(\alpha)};0<\varepsilon\leq 1\}$
 is a bounded  and equicontinuous net of functions on any compact
 set in $\mathbb R^d$, for every $|\alpha|\leq k.$ Thus, by
 Arzel\`{a}-Ascoli theorem, it has a convergent subsequence for
 every $|\alpha|\leq k$ and, by  diagonalization, there exists
 a sequence $(T_{\varepsilon_n})_n$ and $T\in C^k(\mathbb R^d)$ such that
 $T_{\varepsilon_n}^{(\alpha)}\rightarrow T^{(\alpha)}$
uniformly on any compact set $K\subset\mathbb R^d.$ That $||T||_{W^{k,\infty}(\mathbb{R}^{d})}<\infty$ follows now easily. Let
$|\alpha|=k.$
For every $x,y\in \mathbb R^d, \;x\neq
y,$
$$\frac{|T^{(\alpha)}(x)-T^{(\alpha)}(y)|} {|x-y|^{\tau}}=
\lim_{n\rightarrow \infty}\frac{|T_{\varepsilon_n}^{(\alpha)}(x)-T_{\varepsilon_n}^{(\alpha)}(y)|}
{|x-y|^{\tau}}\leq C,
$$
since
\begin{align*}\sup_{x,y\in \mathbb R^d, \;x\neq
y}\lim_{n\rightarrow \infty}\frac{|T_{\varepsilon_n}^{(\alpha)}(x)-T_{\varepsilon_n}^{(\alpha)}(y)|}
{|x-y|^{\tau}}&\leq \sup_{x,y\in \mathbb R^d, \;x\neq
y}\sup_{n\in\mathbb N}\frac{|T_{\varepsilon_n}^{(\alpha)}(x)-T_{\varepsilon_n}^{(\alpha)}(y)|}
{|x-y|^{\tau}}
\\
&
\leq \sup_{x,y\in \mathbb R^d, \;x\neq
y, \varepsilon\leq 1}\frac{|T_{\varepsilon}^{(\alpha)}(x)-T_{\varepsilon}^{(\alpha)}(y)|}
{|x-y|^{\tau}}\leq C,
\end{align*}
and the assertion follows. 
\end{proof}

\end{document}